\documentclass[12pt,a4paper, submit]{article}
\usepackage[english]{babel}
\usepackage{amsmath}
\usepackage{graphicx}
\usepackage{wasysym}
\usepackage{amssymb}
\usepackage[T1]{fontenc}
\usepackage{regexpatch}
\usepackage{amsthm}
\usepackage{tikz}
\usetikzlibrary{matrix}
\usetikzlibrary{automata,topaths}
\usepackage{color}
\usepackage{comment}
\usepackage{polynom}
\usepackage{hyperref}
\usepackage{nameref}
\usepackage{wrapfig}
\usepackage{scalerel,xcolor}
\usepackage{multicol}
\theoremstyle{plain}
\usepackage{mathtools}
\usepackage{listings}
\usepackage{algpseudocode}

\usepackage[style=alphabetic,maxnames=99,maxalphanames=5,maxcitenames  = 3,backend=bibtex]{biblatex}
\usepackage[english]{babel} % Language hyphenation and typographical rules

\usepackage[hmarginratio=1:1,top=32mm,columnsep=20pt]{geometry} % Document margins
\usepackage[hang, small,labelfont=bf,up,textfont=it,up]{caption} % Custom captions under/above floats in tables or figures
\usepackage{authblk}
\usepackage{booktabs} % Horizontal rules in tables
\begin{document}

\newtheorem{thm}{Theorem}[section]

\newtheorem{lem}{Lemma}[section]

\newtheorem{prp} {Proposition}[section]

\newtheorem{dfn}{Definition}[section]

\newtheorem{ex}{Example}[dfn]

\newtheorem{rmk}{Remark}[thm]

\newtheorem{rem}{Remark}[dfn]

\newtheorem{re}{Remark}[section]

\newtheorem{remk}{Remark}[subsection]

\newtheorem{dfnn}{Definition}[subsection]

\newtheorem{exe}{Example}[rmk]

\newtheorem{exempl}{Example}[section]

\newtheorem{pr}{Problem}[section]

\newcommand{\zero}{0}
\newcommand{\pl}{+}
\newcommand{\eg}{=}

\newcommand\F{\mathbb F}
% Print the title
\def\A{\mathbb A}
\def\N{\mathbb N}
\def\K{\mathbb K}
\def\E{\mathcal E}

\title{A new approach in constructing isogenies of elliptic curves in characteristic three}

\author{Marius B\u aloi}
\affil{Faculty of Mathematics and Informatics, University of Bucharest, Academiei st. 14, Bucharest, Romania }
\affil{\textbf{grigore-marius.baloi@s.unibuc.ro}}
\date{}
\maketitle

%----------------------------------------------------------------------------------------
%	ARTICLE CONTENTS
%----------------------------------------------------------------------------------------

\begin{abstract} Given an elliptic curve ${\mathcal E}$ over a field $K$ it is a challenging problem to write down explicit elements of its endomorphism ring ${\rm End}({\mathcal E});$ the problem amounts to find all possible solutions to a functional equation in the field of rational functions $K(X).$ Instead of attempting to describe them directly, we look first for solutions in the larger field of Laurent power series $K((X))$, which we call them {\em formal endomorphisms}. We show that the set of separable formal endomorphisms naturally identifies with a subset of $\frac{1}{X}K[[X]]-$rational points of a plane cubic defined over $K((X)).$ As a by-product, we present a method for finding all formal separable endomorphisms in characteristic $3$. %and an efficient test for determining if a given formal solution is actually rational, yielding to an endomorphism of the given curve.
\end{abstract}
\section{Introduction}

%------------------------------------------------

%------------------------------------------------

%------------------------------------------------

%\section{Computing isogenies in characteristic three}

Let $K$ be a field of characteristic three and let ${\mathcal E}$ be an elliptic curve defined over $K$. Its Weierstrass normal form (WNF, for short) can be (see, e.g. \cite{Silv}) either
\begin{equation}\label{wnf1}
({\rm WNF1})\;\;\;\;Y^2=X^3+AX+B
\end{equation}
 or  
\begin{equation}\label{wnf2}
({\rm WNF2})\;\;\;\;Y^2=X^3+AX^2+B
\end{equation}

\noindent Notice that a curve in WNF2 is automatically non-supersingular (cf e.g. \cite{Silv}, Thm. 4.1, pp 148), so the real case of interest is the WNF1; throughout this paper, we will focus on this case only.
%still, for the sake of completeness, we include also the study of the WNF2 case.
\noindent Let 
$
\varphi : \mathcal{E}\to \mathcal{E}
$
 be a separable isogeny. Then, in affine coordinates, $\varphi$ is of the form
\begin{equation}\label{cohel} 
 \varphi(x,y)=\left( \eta(x), c y \eta'(x)\right)
 \end{equation}
  where $ \eta \in K(X)$ is a rational function and $c\in K^*$.
Indeed, any isogeny must be of the form
$$(x, y)\mapsto (f_1(x,  y), f_2(x, y))$$ with $f_1, f_2$ rational functions.
Since %in both cases
the neutral element is the point at infinity, $\Omega$, of homogeneous coordinates $[x, y, z]=[0,1,0]$, we see that the inverse of a point $P(x, y)$ is $(x, -y)$ and, as isogenies are group morphisms, we see that $f_1$ is fact a rational function on $x,$ say $\eta(x).$ As in both cases, the invariant differential is $\omega=\frac{dx}{y}$, we get that
$c\varphi^*(\omega)=\omega$, for some $c\not=0$. From \cite{Wa}, we get that
$$c\frac{\eta'(x)dx}{f_2(x, y)}=\frac{dx}{y}$$
hence 
$f_2=c\eta'y.$

\hfill

\noindent The aim of the paper is to find an algorithmic way of generating  all isogenies of ${\mathcal E}$.
Generating all isogenies of ${\mathcal E}$ amounts henceforth to find all rational functions $\eta\in K(X)$ that satisfy
\begin{equation}\label{isof}
c^2y^2(\eta')^2=\eta^3+A\eta+B.
\end{equation}
%(or $c^2y^2(\eta')^2=\eta^3+A\eta^2+B$, for the second case of WNF).

\noindent Instead of trying to find all solutions of the above equation in $\eta$, as in \cite{Bostan} for characteristic 2, we enlarge the frame of the field of rational functions to the field   $K((X))$, the fraction field  of the ring $K[[X]]$ of formal power series.
Notice that if $\eta$ is a separable isogeny, it is unramified everywhere, hence in particualar it must have at most  poles of order at most one. Inspired by this, we will call a {\em formal isogeny} any solution %\textcolor{red}{
$\eta$ of \ref{isof} which belongs to $\frac{1}{X}K[[X]].$

%In fact, the main core of the computations boils up in $K[[X]]$ and we can benefit form the fact that this ring is complete; we provide in this way formal solutions to our problem. 

\hfill

%\subsection{Periodicity in $K[[X]].$}

 %A formal power series $S=\displaystyle \sum_{n\geq 0} a_n X^n$ over some finite field $K=\F_q$ is a rational function if and only if it is eventually periodic. For a (eventually) periodic power series $S$ we will denote ${\rm LP}(S)$ the lenght of the period of $S.$ 
 
% \noindent It is a natural question to ask, for a given series $S$ which is rational, $S=\frac{P}{Q}$, to provide an upper bound for ${\rm LP}(S)$ in terms just of degrees of $P$ and $Q$. We can safely assume that $deg(P)<deg(Q).$

%\noindent First, we observe that for any $n\in \N$, 
%\begin{equation}\label{perx} {\rm LP}(S)={\rm LP}\left(\frac{1}{X^n}S\right)\end{equation}

%In the case $\frac{P}{Q}=\frac{1}{X^N-1}$ for some $N\in \N$, one obviously have 

%\begin{equation}\label{perb} {\rm LP}\left(\frac{P}{X^N-1}\right)\leq N.\end{equation}

%This implies that, if $P$ has no multiple roots and no $X-$factor then, 

%\begin{equation}\label{perq}{\rm LP}(\frac{1}{P})\leq q-1\end{equation}

%since $P$ divides $X^{q-1}-1.$

\hfill

%\textcolor{red}{Going along the same lines, one can prove
%\begin{equation}\label{bound}
%LP(1/P)\leq 3^{\lfloor \frac{\log_3(deg(P))}{3}\rfloor +1}
%\end{equation}
%}

\subsection{A splitting of $K((X)).$}
Let $K$ be a field of characteristic three. For any Laurent power series
$
S=\sum_{n\geq k} a_nX^n
$
we will use the decomposition
\begin{equation}\label{split1}
S=\alpha+\beta +\gamma
\end{equation} 
where 
$$\alpha=\sum_{3n+1\geq k} \alpha_nX^{3n+1}, \; \beta=\sum_{3n+2\geq k} \beta_nX^{3n+2}, \;\gamma=\sum_{3n\geq k} \gamma_nX^{3n}.$$
Notice that the above decomposition is given by the splitting 
\begin{equation}\label{split}
K((X))=V_{1}\oplus V_{-1}\oplus V_0
\end{equation}
induced by the formal derivative (denoted by  $'$), where
$$V_0\coloneqq \left\{S\in K((X))\vert S'=0\right\}.$$
$$V_1\coloneqq \left\{S\in K((X))\vert S'=\frac{S}{X}\right\}$$
$$V_{-1}\coloneqq\left\{S\in K((X))\vert S'=-\frac{S}{X}\right\}$$
Sometimes, by an abuse of language, we will called the elements of the above subspaces as "homogeneous" of degrees $0, 1$ and $2$ (or $-1$) respectively.

%Remark also that $V_0$ is in fact a subfield of $K((X))$, canonically isomorphic to $K((X)).$

\noindent Notice also that  the formal power series  $S$ is actually rational, $S\in K(X),$ if and only if $\alpha, \beta$ and $\gamma$ are all rational. Indeed, if $S=\alpha+\beta+\gamma$ one can immediately check by direct computation that one has:
\begin{equation}\label{split2}
\left\{
\begin{array}{lcc}
\alpha=XS'-X^2S''\\
\beta=-X^2S''\\
\gamma= S-XS'-X^2S''.
\end{array}
\right.
\end{equation}

\subsection{\color{black}{A variant of Hensel's lemma}}

\begin{lem}\label{detgamma}Let $A\in K^*$ and $\psi\in V_0\cap K[[X]]$ be arbitrary,  $\psi=\displaystyle\sum_{n\geq 0} C_nx^{3n}.$ Then there exists (and it is unique up to a constant additive factor in $K$) some $\gamma\in V_0\cap K[[X]]$ such that
\begin{equation}\label{gamma3}
\gamma^3+A\gamma=\psi
\end{equation}
if and only if the equation 
\begin{equation}\label{c0} X^3+AX=\psi(0)
\end{equation} has a solution in $K.$

%b) the same conclusion as in a) holds good for the case of equation

%\begin{equation}\label{gamma4}
%\gamma^3+A\gamma^2=\psi
%\end{equation}
%with initial condition
%\begin{equation}\label{c0.2} X^3+AX^2=\psi(0).
%\end{equation}

\end{lem}

	\begin{proof} %The proof is a direct consequence of Hensel lemma, but we present a simple algorithmic solution, since we are interested in the computational aspects.
	a) Letting $\gamma=\displaystyle{\sum_{n\geq 0}\gamma_nX^{3n}}$ we  have	that $$\sum_n \gamma_n^3 X^{9n}+\sum_n \gamma_nX^{3n}=\sum_n C_n X^{3n}.$$ The initial coefficient $\gamma_0$ is determined as a solution of $X^3+AX=\psi(0)$, and for all $n>0$ we have the following recurrence relations:
	
	\begin{equation}\label{recgammacoeff}
	\begin{array}{ccc}
	\gamma_{n}=C_n, \; \forall n\not=3\ell \\
	\gamma_{\ell}^3+\gamma_{3\ell}= C_{3\ell},\; \forall \ell
	\end{array}
	\end{equation}
	Hence the coefficients of $\gamma$ can be determined by the simple linear recurrence from above.
%\textcolor{red}{The proof of b) is similar and shall be omitted.}

\noindent {\bf Remark 1.} It is straightforward that the solution $\gamma$ of \ref{gamma3} is usually just a formal power series, even if the function $\psi$ is rational. An easy example is  given by the case when $\gamma^3+\gamma=X^3$, for which 
$\gamma(X)=\displaystyle\sum_{i=1}(-1)^{i-1}X^{3^i}$  which is not a rational function as the series is not periodic.

\noindent {\bf Remark 2.} Notice that if one weakens the condition that $\psi\in K[[X]]$, allowing principal parts for it, it is possible that the equation \ref{gamma3} has no solution $\gamma$ even in $K((X))$; an immediate example is
$\gamma^3+\gamma=\frac{1}{x^3}.$
	
	\end{proof}

%Taking into account that we are in $char=3$ we get that the (formal) derivatives of the above series are of the form

%\begin{equation}\label{der1}
%\alpha'=\sum_{n\geq k} \alpha_nX^{3n}, \; \beta'=-\sum_{n\geq k} \beta_nX^{3n+1}, \;\gamma'=0
%\end{equation}
%and 

%\begin{equation}\label{der2}
%\alpha"=\gamma"=0, \; \beta"=-\sum_{n\geq k} \beta_nX^{3n}.
%\end{equation}
\section{ The main result}

\begin{thm} Let $K$ be a field of charateristic $3$, and $\E$ be an elliptic curve in WNF1 as in (\ref{wnf1}).  Let $\K$ denote the field $K((X))$ (with decomposition $K((X))=V_0\oplus V_1\oplus V_2$ as in \ref{split}) and let $\A^3_{\K}=\{(\alpha, \beta, \gamma)\vert \alpha, \beta, \gamma \in \K\}$ be the affine $3-$space over $\K.$ %Then the set of all separable  formal endomorphisms of $\E$ identifies with the set of $\K-$rational points of the plane cubic in $\A^3_{\K}$ given by
Let $c\in K, c\not=0$ be arbitrary. Consider the plane cubic ${\rm E}_c$ over $K((X))$ defined by 
\begin{equation}\label{space}
\left\{
\begin{array}{lcc}
c^2Ax\alpha+c^2(x^3+B)\beta=Ax^2\\
c^2(x^3+Ax+B)\left(\frac{\alpha-\beta}{x}\right)^2=(\alpha+\beta+\gamma)^3+A(\alpha+\beta+\gamma)+B
\end{array}
\right.
\end{equation}

a) If $B\not=0$, the set of all separable formal endomorphisms of $\E$ with `` derivative at origin'' equal to $c$ identifies with the set of $K[[X]]-$rational points of ${\rm E}$
such that $\alpha\in V_1,$$ \beta\in V_2$ and $\gamma\in V_0$ and, satisfying the ``compatibility condition'': there exists $k\in K$ such that $k^3+k=c^2B\alpha_1-B$ (where  $\alpha=\alpha_1X+\alpha_4x^4+\alpha_7X^7+\dots).$

b) If $B=0$, then the set of all separable formal endomorphisms of $\E$ with ``derivative at origin'' equal to $c$ identifies with the set of points $(\alpha, \beta, \gamma)$ of ${\rm E}_c$ with $\beta\in \frac{1}{X}V_2$, $\alpha \in V_1, \gamma \in V_0$.
In this case, the compatibility conditions are given by the relations \ref{c1}, \ref{c2}, \ref{c3} and \ref{c4}  below.
%In particular, we see that the set of all separable  endomorphisms of $\E$  identifies with the subset of points which are $K(X)-$rational.
 \end{thm}

\subsection{Proof of the first equation of (\ref{space})}
\begin{lem}
Let $K $ be a field of characteristic $ 3 $ and let $ \mathcal{E} $ be an elliptic curve over $K $ given by $\mathcal{E}:y^2=x^3+Ax+B.$ %(notice that $A\not=0$ otherwise the curve is singular).
 Let $\eta$ be a formal endomorphism of ${\mathcal E}$ defined over $K$ in the form (\ref{cohel}). Write $\eta$ under the form $\eta=\alpha+\beta+\gamma$ as in \ref{split1}. Then 
 $$c^2Ax\alpha+c^2(x^3+B)\beta=Ax^2$$ holds; in particular, we see that the $\alpha-$part determines the $\beta$-part and conversely.%, under simple linear recurrences, and together they determine the $\gamma-$part, still under simple linear recurrencies.	

\end{lem}

\begin{proof}
		
	\noindent We have that 
	\begin{equation} 
	\label{pi:0}
	 c^2y^2(\eta^{'}(x))^2=\eta^3(x)+A\eta(x)+B.  \end{equation}  
	
	\noindent Since $ y^2= x^3+Ax+B$, we have that 
	
	\begin{equation}
	\label{p:1}
	c^2(x^3+Ax+B)(\eta^{'}(x))^2=\eta^3(x)+A\eta(x)+B.
	\end{equation}

	\noindent Taking the derivative of (\ref{p:1}) and using the fact that we are in $char=3$ we have that 
	
	\begin{equation}
	\label{p:2}
	c^2A(\eta^{'}(x))^2+2c^2(x^3+Ax+B)\eta^{'}(x)\eta^{''}(x)=A\eta'(x).
	\end{equation}        
	
	\noindent Since $ \eta $ is separable (by assumption), we have that $ \eta'\neq0 $. Then, dividing (\ref{p:2}) by $ \eta' $, we get that  
	\begin{equation}
	\label{p:3}
		c^2A\eta^{'}(x)+2c^2(x^3+Ax+B)\eta^{''}(x)=A.
	\end{equation}

	\noindent Under the $(\alpha, \beta, \gamma)-$decomposition of $\eta$  we get that
	
	$$c^2A(\alpha'(x)+\beta'(x))+2c^2(x^3+Ax+B)\beta''(x)=A.$$
	and keeping into account that $\alpha\in V_1$ and $ \beta\in V_{-1}$, we further get
	
	$$c^2A\left(\frac{\alpha(x)-\beta(x)}{x}\right)-2c^2(x^3+Ax+B)\frac{\beta(x)}{x^2}=A.$$
	This relation becomes
\begin{equation}\label{wond}
	c^2Ax\alpha(x)+c^2(x^3+B)\beta(x)=Ax^2
\end{equation}
	and, as $c^2A\not=0$ and $c^2(x^3+B)\not =0$, we see that $\alpha$ determines $\beta$ and conversely.
	\end{proof}

\subsection{Proof of the second equation of (\ref{space}) }%(finding the $\gamma-$part)}
\begin{proof}

To retrieve the $\gamma-$part from the $\alpha-$ and $\beta-$ parts, we go back to (\ref{p:1}) getting that
{\small
\begin{equation}
	\label{equ}
	c^2(X^3+AX+B)(\alpha^{'}+\beta')^2=(\alpha+\beta+\gamma)^3+A(\alpha+\beta+\gamma)+B.	
	\end{equation}    }
\noindent  As $\alpha'=\frac{\alpha}{X}$ and $\beta'=-\frac{\beta}{X}$ the above relation becomes:
{\small
\begin{equation}\label{equ1}
	c^2(X^3+Ax+B)\left(\frac{\alpha-\beta}{x}\right)^2=(\alpha(x)+\beta(x)+\gamma(x))^3+A(\alpha(x)+\beta(x)+\gamma(x))+B.	
	\end{equation}  
}    

Proof of a). Choose $\alpha \in V_1\cap K[[X]]$ arbitrary.
As  $B\not=0$, then from equation \ref{wond} we see that $\beta$ is easily determined and, moreover, is belongs to $V_2\cap K[[X]]$ (since $X^3+B$ is invertible in $K[[X]]).$ Now, to retrieve $\gamma$ we use Lemma \ref{detgamma}. In order to apply it, we must first check that 
\begin{equation}\label{psifac}
\psi =c^2(X^3+Ax+B)\left(\frac{\alpha-\beta}{x}\right)^2-\alpha^3-\beta^3-\textcolor{red}{A}\alpha-\textcolor{red}{A}\beta-B
\end{equation}
is in $k[[X]].$ But this is easily verified, since as $\alpha\in V_1\cap K[[X]]]$ and $\beta\in V_2\cap K[[X]]$ we get that $\left(\frac{\alpha-\beta}{x}\right)^2$ also belongs to $K[[X]].$
Next, we need to ceck that $\psi\in V_3\cap K[[X]].$  This follows by looking at the ``homogenous'' components of it and keeping in mind the relation \ref{wond}. Eventually, we need to look at the  ``initial condition'' that asks for the equation $X^3+AX=\psi(0)$ to have a solution in $K;$ this amounts to the existence of a $k\in K $ such that $k^3+k=c^2B\alpha_1 -B$, as stated.

\hfill

Proof of b). In this case, the equation \ref{wond} becomes
$c^2AX\alpha+c^2X^3\beta=AX^2$. So, taking some $\alpha\in V_1\cap K[[X]],$ say $\alpha=X\delta(X^3)$ (with $\delta \in K[[X]]$) we get
$\beta=\frac{AX^2-c^2AX^2\delta}{c^2X^3}=\frac{A-c^2A\delta}{c^2X}$ hence, $\beta=\frac{\beta_{-1}}{X}+S$
where 
\begin{equation}\label{c1}
\beta_{-1}=\frac{A}{c^2}-A\delta_0
\end{equation}
and $S\in X^2K[[X]].$
Now, to use Lemma \ref{detgamma} we must look again at the factor $\psi$ from \ref{psifac}; first, we look at its principal part, which must vanish.
Since, in our case, 
$$\psi =c^2(X^3+AX)\left(\frac{\alpha-\beta}{X}\right)^2-\alpha^3-\beta^3-A\alpha-A\beta,$$
we get that, modulo terms in $K[[X]]$, $\psi$ equals to
$$c^2(X^2+A)\frac{\beta_{-1}^2}{X^3}-\frac{b_{-1}^3}{X^3}-A\frac{b_{-1}}{X}.$$ Then, we obtain that 
\begin{equation}\label{c2}
c^2A\beta_{-1}^2-\beta_{-1}^3=0
\end{equation}
and 
\begin{equation}\label{c3}c^2(\beta_{-1}^2+A\alpha_1\beta_{-1})-A\beta_{-1}=0,
\end{equation}
hence $\beta_{-1}=c^2A$ and $\alpha_1=\frac{c^4-1}{2c^2}.$

Eventually, we look at the condition that $X^3+AX=\psi(0)$ to have a solution.  After direct computations, we get that this amounts to require that 
\begin{equation}\label{c4}
X^3+AX=\psi(0)=2c^2A\beta_{-1}a_2
\end{equation}
to have a solution in $K.$
\end{proof}

\subsection{An algorithm for finding formal endomorphisms}

To summarize the ideas in the previous Theorem, we present an algorithm for finding formal endomorphisms for a given elliptic curve ${\mathcal E}$ of equation $$Y^2=X^3+AX+B$$ (with given `` differential at origin'' $c\in K^\star$) over a field $K$ of characteristic three.

\begin{itemize}
\item Pick any formal power series $\alpha\in V_1$;

    \item Determine $\beta\in V_2$ from equation \ref{wond};

    \item Check the compatibility condition for the choice we made (according to the cases $B\not=0$ or $B=0$); if this is not satisfied, just change the initial coefficient of $\alpha$;

    \item Determine the formal power series $\gamma$ form equation \ref{equ1};

    \item Eventually, the desired formal endomorphism will be given by $(\eta, c\eta')$ where $\eta=\alpha+\beta+\gamma.$

\end{itemize}
 
\section{Worked examples}

\noindent {\bf Example 1.}  Let $K={\mathbb F}_3$ and the elliptic curve be of equation
	$$y^2=x^3+x+1.$$ We want to find out an isogeny whose $\alpha-$part is
 $\alpha(x)=x$ and whose "derivative at the origin" is $c=1.$
	
	\noindent Relation (\ref{wond})
\begin{equation}
	c^2Ax\alpha(x)+c^2(x^3+B)\beta(x)=Ax^2
\end{equation}
becomes
\begin{equation}\label{wondp}
	x\alpha(x)+(x^3+1)\beta(x)=x^2
\end{equation}
which provides that
$\beta=0.$
 To determine $\gamma$, we first pick any $c_0$ such that $c_0^3-c_0=0$; the recurrence for $\gamma$ is given by
relation (\ref{equ1}):
\begin{equation}
c^2x^3\alpha^2+c^2B\alpha^2+c^2Ax\beta^2=x^2(\alpha+\beta+\gamma)^3+x^2A\gamma+Bx^2
\end{equation}
which becomes in this case
$$x^5+x^2=x^2(x+\gamma(x))^3+x^2\gamma(x)+x^2.$$
Keeping into account the initial condition for $\gamma$ given by (\ref{c0}), this immediately implies that
$\gamma(x)=c_0.$
To conclude, all  formal isogenies  as required are of the form
$(x, y)\mapsto (x+c_0, y),$
where $ c_0\in \{0,1,2\} $. 
Notice that they are also isogenies in the usual sense.

\hfill
	
	\noindent {\bf Example 2.}  Let $K={\mathbb F}_3$ and the elliptic curve be of equation
	$$y^2=x^3-x$$ (hence $A=-1$ and $B=0$). We want to find out an isogeny whose $\alpha-$part is
 $\alpha(x)=x $ and whose "derivative at the origin" is $c=1.$
	
\noindent	Relation \ref{wond}
\begin{equation}
	c^2Ax\alpha(x)+c^2(x^3+B)\beta(x)=Ax^2
\end{equation}
becomes
\begin{equation}
	-x\alpha(x)+x^3\beta(x)=-x^2
\end{equation}
that is
\begin{equation}
	-\alpha(x)+x^2\beta(x)=-x
\end{equation}
which implies 
$\beta=0.$
 To determine $\gamma$, we first pick any $c_0$ such that $c_0^3-c_0=0$; the recurrence for $\gamma$ is given by
relation (\ref{equ1}):
\begin{equation}
c^2x^4\alpha(x)^2+c^2A\delta^2(x)=x^3\alpha^3(x)+\delta^3(x)+x^3\gamma(x)^3+Ax^3\gamma(x)
\end{equation}
which becomes in this case
$$x^6+1=x^6+1+x^3\gamma^3(x)-x^3\gamma(x).$$
Keeping into account the initial condition for $\gamma$ given by (\ref{c0})
this imediately imply
$\gamma(x)=c_0.$
To conclude, all  formal isogenies  as required are of the form
$$(x, y)\mapsto (x+c_0, y).$$
Notice that they are also isogenies in the ususal sense.

\hfill

\noindent {\bf Example 3.} In the same setup as in the previous example, suppose we want to describe all the isogenies with $\beta$-part
$\beta(x)=-\frac{1}{x}$
and $c=1.$

\noindent Relation \ref{wondp}
\begin{equation}
	-\alpha(x)+x^2\beta(x)=-x
\end{equation}
implies
$\alpha(x)=0. $
To determine $\gamma$, pick any $c_0$ such that $c_0^3-c_0=0$; the recurrence for $\gamma$ is given by
relation (\ref{equ1}):
\begin{equation}
c^2x^4\alpha(x)^2+c^2A\delta^2(x)=x^3\alpha^3(x)+\delta^3(x)+x^3\gamma(x)^3+Ax^3\gamma(x)
\end{equation}
which becomes in this case
\begin{equation}
1=1+x^3\gamma(x)^3-x^3\gamma(x)
\end{equation}
which implies $\gamma(x)=c_0.$ We get that $\eta=-\frac{1}{x}+c_0, $ that is, the formal isogenies is in this case are of the form
$$(x, y)\mapsto (-\frac{1}{x}+c_0, \frac{y}{x^2}).$$
Notice that they are also rational isogenies as in the previous case.

\hfill

\noindent {\bf Example 4.} Let $ K=\F_9 $ and the elliptic curve be of equation $$y^2=x^3+x+2, $$
(hence $ A=1  $ and $ B=2 $). We want to find an isogeny whose $ \beta- $part is
$\beta(x)=\frac{x^2}{x^9+x^3-1}$
and $c=1.$

\noindent Relation (\ref{wond})
\begin{equation}
x\alpha(x)+(x^3+2)\beta(x)=x^2
\end{equation}
provides 
$\alpha(x)=\frac{x^{10}}{x^9+x^3-1}.$
To determine $\gamma$, pick any $c_0$ such that $c_0^3-c_0=0$; the recurrence for $\gamma$ is given by
relation \ref{equ1}:
\begin{equation}
(x^3+x+2)\left(\frac{\alpha-\beta}{x}\right)^2=(\alpha(x)+\beta(x)+\gamma(x))^3+(\alpha(x)+\beta(x)+\gamma(x))+2
\end{equation}
which produces $$\gamma(x)=\frac{x^6 + x^3 + 1}{x^9 + x^3 +2}.$$ For $ c_0=2, $ we get that $\eta=\frac{x^4+x^2+2x+1}{x^3+x+2} $. We can observe that
$$(x, y)\mapsto \left(\frac{x^4+x^2+2x+1}{x^3+x+2}, -y\cdot\frac{x^6+2x^4+x^3+x^2+x}{x^6 +2x^4 + x^3 + x^2 + x+1}\right).$$
Notice that this isogeny is the multiplication-by-2 map.

\hfill

\noindent {\bf Acknowledgements.} This work was partially supported by a grant of the Ministry of Research, Innovation and Digitalization, CNCS/CCCDI - UEFISCDI, project number ERANET-CHISTERA-IV-PATTERN, within PNCDI IV.

\noindent The author would like to thank V. Vuletescu for asking me the problem and for many valuable suggestions.

%\begin{thebibliography}{100}
%\bibitem{b1} Silverman, "Arithmetic"
%\end{thebibliography}

\printbibliography[heading=bibintoc]

%\printbibliography[heading=bibintoc, title={References}]

\end{document}